\documentclass[reqno]{amsart}
\usepackage{amsmath,amsfonts,amssymb,graphics,graphicx,latexsym,amscd,subfigure,hyperref}

\newtheorem{thm}{Theorem}[section]

\newtheorem{prop}[thm]{Proposition}
\newtheorem{question}[thm]{Question}
\newtheorem{lemma}[thm]{Lemma}
\newtheorem{defn}[thm]{Definition}
\newtheorem{rem}[thm]{Remark}

\newtheorem{exa}[thm]{Example}

\newcommand{\bbR}{\mathbb{R}}
\newcommand{\bbT}{\mathbb{T}}
\newcommand{\bbC}{\mathbb{C}}
\newcommand{\bbZ}{\mathbb{Z}}
\newcommand{\ip}{\cdot}
\newcommand{\Q}{\mathcal{Q}}

\newcommand{\M}{M}
\newcommand{\mcD}{\mathcal{D}}

\begin{document}

\title[]{Hearing Delzant polytopes from the equivariant spectrum}

\author{Emily B. Dryden}
\address{Department of Mathematics, Bucknell University, Lewisburg, PA 17837, USA}
\email{ed012@bucknell.edu}
\author{Victor Guillemin}
\address{Department of Mathematics, Massachusetts Institute of Technology, Cambridge, MA 02139, USA}
\email{vwg@math.mit.edu}
\author{Rosa Sena-Dias}
\address{Centro de An\'alise Matem\'atica, Geometria e Sistemas Din\^amicos, Departamento de Matem\'{a}tica, Instituto Superior T\'{e}cnico, Av. Rovisco Pais, 1049-001 Lisboa, Portugal}
\email{rsenadias@math.mit.edu}

\date{}
\begin{abstract}
 Let $M^{2n}$ be a symplectic toric manifold with a fixed $\mathbb{T}^n$-action and with a toric K\"ahler metric $g$. Abreu \cite{m2} asked whether the spectrum of the Laplace operator $\Delta_g$ on $\mathcal{C}^\infty(M)$ determines the moment polytope of $M$, and hence by Delzant's theorem determines $M$ up to symplectomorphism.  We report on some progress made on an equivariant version of this conjecture. If the moment polygon of $M^4$ is generic and does not have too many pairs of parallel sides, the so-called \emph{equivariant spectrum} of $M$ and the spectrum of its associated real manifold $M_{\bbR}$ determine its polygon, up to translation and a small number of choices. For $M$ of arbitrary even dimension and with integer cohomology class, the equivariant spectrum of the Laplacian acting on sections of a naturally associated line bundle determines the moment polytope of $M$. 
\end{abstract}


\maketitle


\section{Introduction}
Given a Riemannian manifold $(M,g)$, one can consider the Laplace operator $\triangle_g$ acting on the space of smooth functions on $M$. The spectrum of $\Delta_g$ is the set of all eigenvalues of $\Delta_g$ on $\mathcal{C}^\infty(M)$. It is natural to ask
\begin{question}\label{basic_inverse_spec}
 How much about the geometry of the Riemannian manifold $(M,g)$ does the spectrum of the Laplacian $\triangle_g$ determine?
\end{question}
A priori, the answer to this question could be ``The spectrum of the Laplacian determines $(M,g)$.''  However, there are now many examples of Riemannian manifolds with the same spectrum which are not isometric (e.g., \cite{Gor_survey}, \cite{Mil_tori}, \cite{Sunada_method}).  On the other hand, there are also positive results.  For example,  Tanno showed \cite{Tanno} that if $(M^n,g)$ is a compact orientable Riemannian manifold, then for $n \leq 6$ the spectrum of $\Delta_g$ determines whether $(M^n,g)$ is isometric to $(S^n, \text{round})$.    
 
 Translating Question \ref{basic_inverse_spec} into the setting of symplectic toric geometry, Abreu \cite{m2} asked
\begin{question}\label{q:Abreu}
 Let $M$ be a toric manifold equipped with a toric K\"ahler metric $g$. Does the spectrum of the Laplacian $\Delta_g$ determine the moment polytope of $M$?
\end{question}
\noindent In the spirit of Kac \cite{Kac}, this question can be rephrased as
\begin{question}
 Can one hear the moment polytope of a toric manifold?
\end{question}

A toric manifold $M^{2n}$ is a symplectic manifold with a ``compatible'' $\mathbb{T}^n$-action (see \S\ref{sec:toric_geom} for precise definitions).  Such an action determines a moment map from $M$ to $\mathbb{R}^n$ whose image is a convex polytope, called the \emph{moment polytope} or \emph{Delzant polytope} of $M$. It is a well-known theorem in symplectic geometry that the moment polytope of $M$ determines the symplectomorphism type of $M$.
\begin{thm}\label{thm:Delzant}\cite{Delzant}
The moment polytope of a toric symplectic manifold $M$ determines $M$ up to symplectomorphim.
\end{thm}
\noindent Thus, if the answer to Abreu's question is yes, the spectrum of the Laplacian of a symplectic toric manifold determines its symplectomorphism type. 

We examine a modified version of Abreu's question, replacing the spectrum of the Laplacian by what we call the \emph{equivariant spectrum} of the Laplacian. This is simply the spectrum of the Laplacian together with, for each eigenvalue, the weights of the representation of $\bbT^n$ on the eigenspace corresponding to the given eigenvalue.  Question \ref{q:Abreu} then becomes
\begin{question}\label{q:ours}
 Let $M$ be a toric manifold equipped with a toric K\"{a}hler metric $g$. Does the equivariant spectrum of $\Delta_g$ on $\mathcal{C}^\infty(M)$ determine the moment polytope of $M$?
\end{question}

We will use heat invariant techniques to study this question.  Given a Riemannian manifold $(M,g)$, let $\text{Spec}(M)$ be the set of eigenvalues of $\Delta_g$.  When $M$ is compact, a fundamental solution of the heat equation, or \emph{heat kernel}, is uniquely determined.  The trace of the heat kernel $K(t,x,y)$ is given by $Z(t) = \int_M K(t,x,x) dx$, satisfies
\[
Z(t) = \sum_{\lambda \in \text{Spec}(M)} e^{-\lambda t},
\]
and has an asymptotic expansion as $t$ goes to zero; this expansion yields \emph{heat invariants}, which have proven to be an important tool in studying inverse spectral problems related to Question \ref{basic_inverse_spec}.  For example, they show that geometric quantities such as the volume, the dimension, and certain quantities involving the curvature of $M$ are determined by $\text{Spec}(M)$.      

In the present setting, the torus action gives a family of isometries of $M$; Donnelly \cite{Don1} gives an asymptotic expansion of the heat trace in the presence of an isometry, and we will use this expansion to glean geometric data from the equivariant spectrum (see \S\ref{sec:heat}).  
The leading-order term appearing in Donnelly's formula depends on the dimension of the fixed point set of the isometry considered. We will see that this is largest when the isometry corresponds to an element in the torus which is perpendicular to a facet of the moment polytope of $M$, where we are identifying the torus with the dual of its Lie algebra.  Thus the equivariant spectrum tells us when an element in the torus is perpendicular to a facet; moreover, we can recover the volumes of facets from the coefficient of the leading-order term in Donnelly's expansion.

Combining these ideas with combinatorial and geometric arguments and the usual heat invariants for the real manifold $\M_{\bbR}$ naturally associated to a toric manifold $M$, we will prove that we can hear many Delzant polygons.
\begin{thm}\label{equi_abreu's}
Let $M^4$ be a toric symplectic manifold with a fixed torus action and a toric metric. Given the equivariant spectrum of $M$ and the spectrum of $M_{\bbR}$, we can reconstruct the moment polygon $P$ of $M$ up to two choices and up to translation for generic polygons with no more than $2$ pairs of parallel sides.
\end{thm}

\begin{rem}
Even though we require knowledge of the equivariant and real spectra and a fixed $\bbT^2$-action, we can recover the actual Delzant polygon, up to translation and two choices.  In the original version of the question, the polygon is necessarily recovered only up to an $SL(2, \mathbb{Z})$-transformation.
\end{rem}

Finally, we show that if we consider the Laplacian acting on sections of a line bundle naturally associated to our symplectic toric manifold, we can hear the Delzant polytope.

The paper is organized as follows.  In \S \ref{sec:toric_geom} we give the necessary background from symplectic geometry, including a thorough treatment of fixed point sets of the torus action on a symplectic manifold.  Donnelly's theorem and its consequences are presented in \S \ref{sec:heat}.  
We explore the relationship between the combinatorial constraints of Delzant polygons and their geometry in \S \ref{sec:comb}, and describe in detail the polygons to which our results apply.  Then, in \S \ref{sec:zoo}, we examine how ``frequent'' these polygons are among the set of all Delzant polygons.  By replacing the equivariant spectrum of the Laplacian acting on functions by the equivariant spectrum of a natural line bundle associated to our toric manifold, we obtain results in arbitrary even dimension in \S \ref{sec:bundle}.  We end with some concluding remarks.

\vspace{.5cm}
\noindent \textbf{Acknowledgments:} We are very grateful to Ana Rita Pires for her enthusiasm and insightful formulation of the statement of Lemma \ref{most_obtuse}.  The first and third authors appreciate the hospitality shown to them by the Mathematics Department at MIT during their visits there.  We would also like to thank the referee for a careful reading of and helpful comments on an earlier version of this paper.


\section{Some toric geometry}\label{sec:toric_geom}

\subsection{Background}

We begin by recalling some definitions and well-known facts related to toric manifolds.  For more details and background on symplectic and toric geometry, a good general reference is \cite{CdaS}.  

\begin{defn}
 A \emph{symplectic toric manifold} $\M^{2n}$ is a compact connected symplectic manifold $(\M,\omega)$ with an effective Hamiltonian $\bbT^n$-action.
\end{defn}
Such an action has a corresponding moment map $\phi:\M \rightarrow \mathbb{R}^n$, defined up to translations in $\bbR^n$, where we have identified $\mathbb{R}^n$ with its dual. This moment map depends on the symplectic form $\omega$ but its image (up to translation) does not. It is a convex polytope in $\mathbb{R}^n$ of Delzant type.  
\begin{defn}
 A convex polytope $P$ in $\mathbb{R}^n$ is \emph{Delzant} if
 \begin{enumerate}
 \item there are $n$ edges meeting at each vertex;
 \item for every facet of $P$, a primitive outward normal can be chosen in $\bbZ^n$;
  \item for every vertex of $P$, the outward normals corresponding to the facets meeting at that vertex form a basis for $\bbZ^n$.
\end{enumerate}
\end{defn}

\begin{exa}\label{exa:CP2}
(cf. \cite[p. 173]{CdaS}) Consider the manifold $\mathbb{CP}^2$ equipped with the Fubini-Study form $\omega_{FS}$.  A $\bbT^2$-action on $\mathbb{CP}^2$ is given by
\[
(e^{i \theta_1}, e^{i\theta_2}) \cdot [z_0,z_1,z_2] = [z_0, e^{-i\theta_1}z_1, e^{-i\theta_2}z_2]
\]
with moment map
\[
\phi[z_0,z_1,z_2] = \frac{1}{2} \left( \frac{|z_1|^2}{|z_0|^2 + |z_1|^2 + |z_2|^2}, \frac{|z_2|^2}{|z_0|^2 + |z_1|^2 + |z_2|^2} \right).
\]
The moment polygon $P=\phi(\mathbb{CP}^2)$ is shown in Figure~\ref{fig:CP2}.  One easily checks that it is Delzant.  Note that if we were to define a different $\bbT^2$-action on $\mathbb{CP}^2$ by
\[
(e^{i \theta_1}, e^{i\theta_2}) \cdot [z_0,z_1,z_2] = [z_0, e^{i\theta_1}z_1, e^{i\theta_2}z_2]
\]
then the moment polygon for this action would be $-P$, i.e., the rotation of $P$ about the origin by $\pi$.
\end{exa}

\begin{figure}[h]
\centering
\begin{center}
\setlength{\unitlength}{5mm}
\begin{picture}(5.5,5)(0,0)
\put(0.5,1){\line(1,0){3}}
\put(0.5,1){\line(0,1){3}}
\put(0.5,4){\line(1,-1){3}}
\put(0,0.2){(0,0)}
\put(0,4.3){(0,$\frac{1}{2}$)}
\put(3.8,0.7){($\frac{1}{2}$,0)}
\end{picture}
\end{center}
\caption{The moment polygon $P=\phi(\mathbb{CP}^2)$}
\label{fig:CP2}
\end{figure}
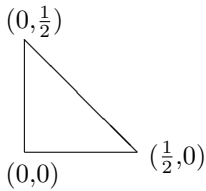

Given a moment polytope $P$ in $\mathbb{R}^n$, Delzant \cite{Delzant} has given a canonical way to associate to it a symplectic manifold $(\M_P,\omega_P)$ together with an effective Hamiltonian torus action $\tau_P$ with moment map $\phi_P$ such that $\phi_P(M_P) = P$; in fact, this is a bijective correspondence.  Moreover, Delzant proved that the moment polytope of a toric manifold determines its symplectic type.
\begin{thm}\cite{Delzant}
 Every toric manifold whose Delzant polytope is $SL(n,\bbZ)$-equivalent to $P$ is equivariantly symplectomorphic to $\M_P$.
\end{thm}

\noindent Note that we say that two Delzant polytopes $P$ and $P'$ are $SL(n,\bbZ)$-equivalent if there exists $A \in SL(n, \bbZ)$ such that $P' = AP$ as sets.

In addition to the symplectic structure associated to a Delzant polytope $P$, there is also a complex structure $J_P$ associated to $P$ via a natural construction (see \cite{Delzant}, \cite{g1}). This complex structure is invariant under the torus action.  Thus $P$ determines both a symplectic and a complex structure of the associated K\"{a}hler toric manifold, and together these structures determine a torus-invariant Riemannian metric $g_P$. The triple $(\omega_P,J_P,g_P)$ is a K\"ahler structure on the manifold $M_P$ called the \emph{reduced K\"ahler structure}. Taking the symplectic point of view, one gets other torus-invariant Riemannian metrics by starting with a fixed symplectic manifold $(M_P, \omega_P)$ and considering all complex structures on $M_P$ that are compatible with $\omega_P$ and invariant under the torus action.  For such a complex structure $J$, a metric is given by $g(X,Y) = \omega_P(X,JY)$.   
Viewing toric manifolds from the perspective of complex geometry, one starts with a fixed complex manifold $(M_P, J_P)$ and considers all torus-invariant symplectic forms on $M_P$ that are compatible with $J_P$.  For such a symplectic form $\omega$, a torus-invariant metric is given by $g(X,Y) = \omega(X,J_P Y)$.  
It is possible to translate between the symplectic and complex perspectives, and it turns out that both perspectives give rise to the same set of torus-invariant Riemannian metrics.  We refer to these metrics as \emph{K\"ahler toric metrics} or \emph{toric metrics} for short. 
Given a Delzant polytope $P$, the reduced K\"ahler structure $(\M_P,\omega_P,J_P)$ gives a way to build a toric metric from $\omega_P$ and $J_P$. This metric is called the reduced metric and it has been completely determined in \cite{g1}. In \cite{m1} and \cite{m2}, Abreu has shown how to characterize all other toric metrics on $M_P$ using the reduced metric.

We are interested in the spectrum of the Laplacian on a symplectic toric manifold with any such toric metric $g$; the torus action associates some natural additional data to the spectrum. To be more precise, we denote by $\psi:\bbT^n\rightarrow Sympl(M)$ the group homomorphism given by the Hamiltonian torus action. Note that a metric is toric exactly when $\psi(e^{i\theta})$ is an isometry for all $\theta \in \bbR^n$. For each $\theta\in \bbR^n$ and each eigenvalue $\lambda$ of the Laplacian on $(M, g)$, $\psi(e^{i\theta})$ induces an action on the eigenspace corresponding to $\lambda$. This action splits according to weights. 
\begin{defn}
Let $\M$ be a toric manifold with a fixed torus action; denote by $\psi:\bbT^n\rightarrow Sympl(\M)$ the corresponding group homomorphism, and let $g$ be a toric metric on $M$. The \emph{equivariant spectrum} is the list of all the eigenvalues of the Laplacian on $(M,g)$  together with the weights of the action induced by $\psi(e^{i\theta})$ on the corresponding eigenspaces, for all $\theta \in \bbR^n$.  The eigenvalues and weights are listed with multiplicities.        
\end{defn}

\subsection{Fixed point sets}
The goal of this subsection is to study the fixed point sets of the isometries $\psi(e^{i\theta})$; these results will be used in the calculation of the heat invariants. The results that follow are well known but we give proofs for the sake of completeness.  We start by recalling Delzant's construction (see \cite{g1} for more details). 

Let $P$ be a Delzant polytope with $d$ facets. Consider the following exact sequences
\begin{eqnarray}\label{exactseq}
0 \rightarrow N \rightarrow \bbT^d \xrightarrow{\beta'} \bbT^n \rightarrow 0 &,& \\
0  \rightarrow  \mathfrak{n} \xrightarrow{\iota} \bbR^d \xrightarrow{\beta} \bbR^n \rightarrow 0 &.&
\end{eqnarray}
Here $\beta: \bbR^d \rightarrow \bbR^n$ is given by $ \beta(e_i)=u_i$, where $\{e_i\}$ is the canonical basis for $\bbR^d$ and $u_i$ denotes the primitive outward normal to the $i$th facet of the polytope; $\mathfrak{n}$ is the Lie algebra of $N$.  The group $N$ acts symplectically on $\bbC^d$ with moment map
\begin{displaymath}
 h(z)=\sum|z_i|^2\iota^*e_i,
\end{displaymath}
where $\iota^*$ is dual to $\iota$.
The toric manifold associated to $P$ is
\begin{displaymath}
 \M=h^{-1}(c)/N
\end{displaymath}
where $c \in \mathfrak{n}^*$.  We denote the projection map from $h^{-1}(c)$ to $\M$ by $\pi$. The torus $\bbT^d$ acts on $\bbC^d$ and therefore $\bbT^d/N$ acts on $\M$ by 
\begin{displaymath}
 e^{i\theta}\ip [z_1,\ldots, z_d]=[e^{i\theta_1}z_1,\ldots,e^{i\theta_d}z_d],
\end{displaymath}
where $\theta=(\theta_1,\ldots,\theta_d) \in \bbR^d$. The map $\beta'$ gives an identification
$\bbT^d/N\rightarrow \bbT^n$; using this identification we see that $\bbT^n$ acts on $\M$ and, for example,
\begin{displaymath}
 e^{itu_l}\ip [z_1,\ldots, z_d]=[z_1,\ldots,e^{it}z_l,\ldots,z_d].
\end{displaymath}

The usual involution of $\bbC^d$, namely $\sigma(z_1,\ldots,z_d)=(\bar{z_1},\ldots,\bar{z_d})$, descends to an involution of $M$. The fixed point set of $\sigma$ is what we refer to as the real manifold associated to $M$.
\begin{defn}\label{defn:realmfld}
The \emph{real manifold associated to $M$}, denoted by $M_{\bbR}$, is the fixed point set of the involution induced on $M$ by the usual involution on $\bbC^d$.
\end{defn}

Assume we endow $\bbC^d$ with its usual symplectic structure $\frac{i}{2}\sum dz_i \wedge d\bar{z_i}$. The quotient construction preceding Definition \ref{defn:realmfld} gives a symplectic form $\omega_P$. 

We may also construct $M$ as a space with a $\bbT_\bbC^n$-action where $\bbT_\bbC^n$ is the complex torus; see \cite{g1} for details. The advantage of this viewpoint is that it shows that $M$ is also a complex manifold.  The complex structure $J_P$ thus obtained is compatible with the symplectic form on $M$ and together these determine the reduced metric $g_P(X,Y) = \omega_P(X,J_P Y)$ on $M$. The moment map with respect to $\omega_P$ for the $\bbT^n_{\bbC}$-action on $\M$, $\phi_P:\M\rightarrow (\bbR^n)^*$, is given as follows. Consider 
\begin{displaymath}
 \bbC^d\xrightarrow{h} (\bbR^d)^*\xrightarrow{p} (\bbR^d)^*/\mathfrak{n} \xleftarrow{\ \beta^*} (\bbR^n)^*.
\end{displaymath}
The moment map for the $\bbT^n_{\bbC}$-action is given by 
\begin{displaymath}
p\circ h=\beta^*\circ\phi_P \circ \pi.
\end{displaymath}
Therefore
\begin{displaymath}
\left< \phi_P [z_1,\ldots, z_d],u_l \right> = \left< \phi_P [z_1,\ldots, z_d],\beta e_l \right> = \left< \beta^*\phi_P [z_1,\ldots,  z_d],e_l \right> =|z_l|^2+\lambda_l, 
\end{displaymath}
where the $\lambda_l$ correspond to a choice of a constant in the moment map.

Now suppose $M$ has a different toric metric on it compatible with the given complex structure. This metric can be seen to be associated with a different symplectic form, and Delzant's theorem gives a way to relate the two symplectic structures.  We will examine this in more detail in the proof of Lemma \ref{fixed_point_set}, which describes the fixed point sets of $\psi(e^{i\theta})$ for various $\theta \in \bbR^n$. Before stating the lemma, we make some conventions.
\begin{defn}
Given a face $F$ of codimension $r+1$ sitting inside a face $F'$ of codimension $r$ in a convex polytope $P$, we say that a vector is \emph{normal} to $F$ in $F'$ if it is in the linear subspace determined by $F'$ and is orthogonal to the linear subspace determined by $F$.
\end{defn}
\begin{rem}
A vertex in a polytope has maximal codimension in that polytope, and every vector is normal to it.
An $n$-dimensional polytope in $\bbR^n$ is itself a face of codimension $0$ whose only normal is zero.
The normal to a facet is the normal to that facet in the whole polytope. 
\end{rem}
We are now in a position to state our lemma.
\begin{lemma}\label{fixed_point_set}
 Let $\theta\in \bbR^n$. The fixed point set of $\psi(e^{i\theta})$, denoted $F_\theta$, is the union of the pre-images via the moment map of all faces to which $\theta$ is normal in a face of lower codimension.
\end{lemma}

\begin{proof}
More specifically, we prove the following statements.
\begin{itemize}
\item If $\theta$ is generic then $F_\theta$ is the pre-image via the moment map of the vertices of $P$.
\item If $\theta$ is a (nonzero) multiple of a normal to a facet of $P$, then $F_\theta$ contains the union of the pre-image via the moment map of the corresponding facet with the pre-image of the vertices of $P$.
\item Let $F$ be a face of codimension $(r+1)$ sitting inside a face $F'$ of codimension $r$. Let $\theta$ be a (nonzero) multiple of the normal to $F$ in $F'$.  Then $F_\theta$ contains the union of the pre-image of $F$ via the moment map with the pre-image of the vertices of $P$.
\item Given any $\theta$, $F_\theta$ is the union of the pre-images via the moment map of all faces to which it is normal with respect to a face of lower codimension.
\end{itemize}

This result depends on the metric on $M$ via the moment map. We first treat the case when the metric is reduced and no vector is normal to more than one face (ignoring vertices). We begin by noting that from the characterization of the moment map $\phi_P$ given by
\begin{displaymath}
\langle \phi_P[z_1,\ldots, z_d],u_l\rangle=|z_l|^2+\lambda_l,
\end{displaymath}
we see that the pre-image of a facet $\{x\in\bbR^n: x\ip u_l-\lambda_l=0\}$ in $P$ via the moment map is 
\begin{displaymath}
\{[z_1,\ldots, z_d]\in \M : z_l=0\}.
\end{displaymath}
In the same way we can describe the pre-image of any face of positive codimension. For example, the pre-image via the moment map of the vertex where the first $n$ edges meet is $[0,\ldots,0, z_{n+1},\ldots, z_d]$. It is fixed by $u_l$ for $l \in \{1,\ldots,n\}$. 

Given $u \in \bbR^n$, the Delzant condition implies that $u$ can be written as a linear combination $u = \alpha_1 u_1+\cdots+\alpha_n u_n$, where $u_1, \ldots, u_n$ are the primitive outward normals to the facets meeting at a given vertex and the $\alpha_i$ are real numbers.  Thus
\begin{displaymath}
e^{iu} \ip [z_1,\ldots z_d]=[e^{i\alpha_1}z_1,\ldots,e^{i\alpha_n}z_n,z_{n+1},\ldots,z_d]
\end{displaymath}
and so the $u$-action always fixes points of the form $[0,\ldots,0, z_{n+1},\ldots, z_d]$, i.e., the pre-image of the vertex. The same reasoning applies to any vertex. We also see from this that for generic $u$ (that is, for generic $\alpha_i$) there are no other fixed points under the $u$-action. This proves the first assertion in the lemma.

We also have
\begin{displaymath}
e^{itu_1}\ip [z_1,\ldots, z_d]=[e^{it}z_1,z_{2},\ldots, z_d].
\end{displaymath}
Since $e_1$ is not in $N$, $[e^{it}z_1,z_{2},\ldots, z_d] = [z_1,\ldots, z_d]$ in $M$ for all $t$ exactly when $z_1=0$. So the fixed point set of $e^{itu_1}$ is the pre-image of the first facet. In general, the fixed point set of $e^{itu_l}$ is the pre-image of the $l$th facet.  This proves the second assertion. 

We will now use the second assertion to prove the third one. Consider a face $F'$ which is at the intersection of facets labeled $i_1,\ldots,i_r$, and let $G \subset \bbT^n$ be the sub-torus such that $\mathfrak{G}^* =  \{u_{i_1},\ldots, u_{i_r}\}^{\perp}$ is the dual of its Lie algebra. This sub-torus acts on the pre-image via the moment map of $F'$, denoted $\M_{F'}$, making $\M_{F'}$ into a toric manifold with moment map $\phi_{F'}: M_{F'} \rightarrow \mathfrak{G}^*$. There is an injective map
\begin{displaymath}
 \iota_{F'}: \mathfrak{G}^*\rightarrow (\bbR^n)^*
\end{displaymath}
and we define $\widetilde{\phi_{F'}}:= \iota_{F'} \circ \phi_{F'}$. It is clear that $\phi_P|_{\M_{F'}}=\widetilde{\phi_{F'}}$. The second assertion implies that if $n$ is normal to one of the facets of the image of $\phi_{F'}$, then the fixed point set of the $S^1$-action generated by $n$ in $G$ is the union of the pre-image via $\phi_{F'}$ of the facet with the pre-image of the vertices of $\phi_{F'}(M_{F'})$. Note that the vertices of $\phi_{F'}(M_{F'})$ are contained in the vertices of $P$.  Thus the fixed point set of the image of $n$ via $\iota_{F'}$
in $\bbR^n$ (which we are identifying with its dual) is $\phi_{F'}^{-1}$ of a facet in $\phi_{F'}(\M_{F'})$, i.e., $\phi_P^{-1}$ of a facet in $\phi_F'(\M_{F'})$, together with the pre-image of the vertices of $P$. The image of $n$ mentioned above is of course the normal to $F$ in $F'$. 

Now suppose that a vector $\theta$ is normal to more than one face.  The above arguments can be applied to each of these faces to obtain the last assertion.
Finally, suppose $M$ has an arbitrary toric metric on it compatible with the given complex structure, and thus associated with a different symplectic form $\omega$. The two symplectic structures are related via the commutative diagram
$$\begin{CD}
M  @>{\eta}>>  M_P\\
@V{\phi}VV      @VV{\phi_P}V \\
P    @>{\operatorname{id}}>> P
\end{CD}$$
with $\eta^*\omega_P=\omega$.  The function $\eta$ is $\bbT^n$-equivariant, i.e., $\eta(t.x)=t\eta(x)$ for all $t \in \bbT^n$.  Thus  
\begin{displaymath}
\eta(F_{M,\theta})=F_{M_P,\theta}.
\end{displaymath}
We determined the set $F_{M_P,\theta}$ and how it relates to the moment map $\phi_P$ in the preceding arguments, so the commutative diagram gives the desired result.
\end{proof}

To complete our discussion of the fixed point sets of the isometries $\psi(e^{i\theta})$, we give the relationship between the volume of a face in $P$ and the volume of its pre-image under the moment map.
\begin{lemma} \label{volumes_of_pre_images}
Consider a face $F$ of dimension $q$ in the Delzant polytope $P$ of a symplectic toric manifold $M$ endowed with a symplectic form $\omega$. Let $\phi$ be the moment map of the torus action with respect to the form $\omega$. Then
\begin{displaymath}
\text{vol}_\omega(\phi^{-1}(F))=(2\pi)^q\text{vol}(F).
\end{displaymath}
\end{lemma}
\begin{proof}
The key point is that there are symplectic coordinates on an open dense set of $M$. Namely, this dense set can be viewed as $\mathring{P} \times \bbT^n$, where $\mathring{P}$ is the interior of $P$.  Then $\phi$ determines coordinates $x$ on $\mathring{P}$, and there are coordinates $v$ on $\bbT^n$.  The volume of $M$ is given by
\begin{displaymath}
\int_{\mathring{P}\times \bbT^n}(dx\wedge dv)^n=\int_P(dx)^n\int_{\bbT^n}(dv)^n=(2\pi)^n\text{vol}(P).
\end{displaymath}
In the same way there are symplectic coordinates on an open dense subset of the pre-image of a $q$-dimensional face and one can argue as above. One can also use the fact that the pre-image of a face is itself toric with moment map the restriction of $\phi$. 
\end{proof}


\section{Heat invariants in a nutshell}\label{sec:heat}

Donnelly \cite{Don1} gave an asymptotic expansion of the heat trace in the presence of an isometry; we will use this expansion to obtain geometric information from the equivariant spectrum of a toric manifold.  
For general background on heat invariants, good references are \cite{BGM} and \cite{Chavelbook}.
We begin by recalling the setting in which we work.

Let $(M^{2n}, \omega)$ be a symplectic toric manifold equipped with a fixed effective $\bbT^n$-action and corresponding group homomorphism $\psi:\bbT^n\rightarrow Sympl(M)$, and with a toric metric $g$.  Fix $\theta \in \bbR^n$.  Then $\psi(e^{i\theta})$ is an isometry of $M$,  and for each eigenvalue $\lambda$ of the Laplacian on $(M,g), \psi(e^{i\theta})$ induces a representation on the eigenspace corresponding to $\lambda$ which we denote by $\psi^\sharp_\lambda(\theta)$.  Let $F_\theta$ denote the fixed point set of $\psi(e^{i\theta})$.  
For a fixed component $\Q$ in $F_\theta$ and a fixed point $a$ in this component, $\psi(e^{i\theta})$ induces an action $A: (T_a\Q)^\perp \rightarrow (T_a\Q)^\perp$; note that $(Id - A)$ is invertible, and let $B$ denote $(Id - A)^{-1}$.   

Using this notation, we can now state Donnelly's theorem as it applies to our setting.

\begin{thm}\cite{Don1}\label{thm:spectral_data}  
There is an asymptotic expansion as $t \downarrow 0$ given by 
\begin{equation}\label{eqn:don}
\sum_\lambda \text{tr} ( \psi^\sharp_\lambda(\theta))e^{-t\lambda}  \simeq 
\sum_{\Q \subset F_\theta} (4 \pi t)^{-q/2} \sum_{k=0}^{\infty} t^k \int_{\Q} b_k (\theta, a) \text{dvol}_{\Q} (a)
\end{equation}
where $q$ is the dimension of the component $\Q$ of $F_\theta$, $b_i(\theta,a) = |\det B| b_{i}'(\theta, a)$ and $b_i'(\theta,a)$ is an invariant polynomial in the components of $B$ and the curvature tensor $R$ of $M$ and its covariant derivatives.
\end{thm}

Note that the equivariant spectrum determines the left side of \eqref{eqn:don}. Hence we seek to determine what the right side of \eqref{eqn:don} tells us about the toric geometry of our manifold.   For our purposes, we will only need to compute $b_0$; Donnelly showed that $b_0'(\theta,a) = 1$, hence our heat invariants will depend solely on $B$.  Computing the matrix $B$ and examining the coefficient corresponding to the leading term in the asymptotic expansion will allow us to prove the following proposition, which is key to the proof of Theorem \ref{equi_abreu's}.

\begin{prop}\label{prop:sd}
Let $(M^{2n}, g)$ be a toric manifold with a fixed torus action and corresponding group homomorphism $\psi:\bbT^n\rightarrow Sympl(M)$, where $g$ is a toric metric.   
For each non-generic $\theta \in \bbR^n$, the equivariant spectrum determines a volume and an (unsigned) normal vector corresponding to the face(s) of minimal codimension associated to $F_\theta$.  If there is a unique face of minimal codimension, this volume is precisely the volume of that face; else, the volume is the sum of the volumes of the parallel faces. 
\end{prop}

\begin{rem}
The reader may want to keep the case $n=2$ in mind, as that will be the setting of our present application of this result.  For $n=2$ the proposition says that the equivariant spectrum determines the (unsigned) normal vectors to the edges of the Delzant polygon and the sum of the lengths of the edges corresponding to each normal vector.
\end{rem}

\begin{proof}
We use Theorem \ref{thm:spectral_data}.  In particular, we calculate $|\det B|$ in the case when $\theta \in \bbR^n$ is non-generic.  Let $F$ be a face of minimal codimension $q$ associated to $F_\theta$, say with normal $u_F$.  Let $Q$ be the pre-image under the moment map of $F$; the dimension of $Q$ is $2(n-q)$. Then the action $A$ on the fiber of the normal bundle over a point $a \in Q$ takes the form
\[
A = 
\begin{pmatrix}
A_1 & 0 & \cdots & 0  \\
0 & A_2 & \cdots & 0 \\
\vdots & \vdots & \ddots & \vdots\\
0 & 0 & \cdots & A_q
\end{pmatrix}
\]    
where each $A_i$ is of the form
\[
A_i = 
\begin{pmatrix}
\cos( \alpha_i (\theta)) & -\sin( \alpha_i (\theta))  \\
\sin( \alpha_i (\theta)) & \cos( \alpha_i (\theta)) \\
\end{pmatrix}.
\]    
Here $\alpha_i(\theta)$ is the weight of the action in the direction of $u_F$.  Thus 
\[
|\det B| = \frac{1}{\prod_{i=1}^q (2 - 2 \cos( \alpha_i(\theta)))}
\] 
and for $\theta \in \bbR^n$ equal to a multiple of $u_F$, the leading term in \eqref{eqn:don} is
\begin{equation}\label{eqn:leading}
(4 \pi t)^{-(n-q)} \frac{\text{vol}(Q)}{\prod_{i=1}^q (2 - 2 \cos( \alpha_i(\theta)))}.
\end{equation}
Note that the contributions of the other components of $F_\theta$ correspond to higher powers of $t$ than the power of $t$ in this term.  In particular, the pre-images of the vertices contribute to the constant term, so that the normal directions are ``hotter'' than the generic directions.  Also, the assumption that $F$ is a face of minimal codimension associated to $F_\theta$ implies that there will not be contributions to the leading term coming from the pre-images of other faces associated to $F_\theta$.   
Substituting $s\theta$ for $\theta$ in \eqref{eqn:leading} and noting that $\alpha_i(s\theta)=s\alpha_i(\theta)$ for all $i\in \{1,\ldots, q\}$, we know 
\begin{displaymath}
\frac{\text{vol}(Q)}{\prod_{i=1}^q (2 - 2 \cos(s \alpha_i(\theta)))}
\end{displaymath}
for all values of $s \in \bbR$.  Hence we know $\text{vol}(Q)$.   

As we have seen in Lemma \ref{volumes_of_pre_images}, one can relate the volume of $Q$ to the volumes of the faces in the polytope which correspond to it under the moment map.  If $F$ is the only face of minimal codimension $q$ associated to $F_\theta$, then $\text{vol}(Q)=(2\pi)^{n-q}\text{vol}(F)$ and we know the volume of $F$ exactly; else, we have this relationship for each face and its corresponding pre-image under the moment map and we know the sum of the volumes of the faces of minimal codimension which are normal to $u_F$.

We may also take $\theta \in \bbR^n$ equal to zero, so that $\psi(e^{i \theta})$ is the identity.  In this case we get the usual asymptotic expansion of the heat trace, and thus we obtain the usual heat invariants.  In particular the spectrum determines the volume of $M$, so by Lemma \ref{volumes_of_pre_images} we hear the volume of $P$.
\end{proof}

When $n=2$, we get additional information about our polygon from the spectrum of the real manifold $M_{\bbR}$ naturally associated to $M$.

\begin{prop}\label{number_facets}
Let the setup be as in Proposition \ref{prop:sd}. If $n=2$, the spectrum of $M_{\bbR}$ determines the number of vertices of $P$.
\end{prop}

We postpone the proof of this proposition until \S \ref{sec:zoo}, after the necessary background has been motivated and explained.


\section{Constructing polygons}\label{sec:comb}

We now examine to what extent the geometric data provided by the equivariant and real spectra determine a Delzant polygon.  
In the two-dimensional case the data provided by these spectra as in Propositions \ref{prop:sd} and \ref{number_facets} reduces to
\begin{enumerate}\label{sd2}
 \item the number of edges;
 \item the set of (unsigned) normal vectors to the edges, denoted $\{u_1, \ldots, u_r\}$;
 \item the sums of the lengths of the edges with normal vector $u_i$, denoted $l_i$, for each $i\in\{1,\ldots,r\}$;
 \item the volume of the polygon.
\end{enumerate}
We will refer to this collection of data as $\mcD$.

Note that each (unsigned) normal determines a family of parallel lines, with a corresponding edge lying along one of the lines in the family.  Specifying the position of a vertex of an edge then determines the line.  If we also know the length of the edge, we know the edge vector up to sign.  The following lemma gives a procedure for constructing a convex polygon from a specified set of edge vectors.

\begin{lemma}\label{most_obtuse}
Given a set of $d$ vectors in $\bbR^2$ that are known to be the edges of a convex polygon, an arbitrary element of the set (label it $e_1$), and a signed normal to $e_1$ (labeled $u_1$), there is a unique convex polygon $P_{e_1, u_1}$ satisfying 
\begin{itemize}
\item $0\in P_{e_1, u_1}$;
\item $e_1 \in P_{e_1, u_1}$;
\item  $u_1$ points outward from $P_{e_1, u_1}$;
\item the $d$ edge vectors form the edges of $P_{e_1, u_1}$.
\end{itemize}
\end{lemma} 
These conditions say that once $e_1$ and $u_1$ are chosen, there is only one ordering of the set of edge vectors that gives rise to a convex polygon.
\begin{proof}
We place the initial vertex of the polygon at $0$ and the second vertex at the terminal point of $e_1$.  We may then arrange the remaining $d-1$ edge vectors so that the $d$ vectors form a convex polygon; the edge vector based at the second vertex is labeled $e_2$, the edge vector based at the terminal point of $e_2$ is labeled $e_3$, and so on.  We want to show that this ordering of the edge vectors is unique.  To do so, we use the ``most obtuse'' angle idea: we determine all the edge vectors that are in the half-plane $\mathcal{H}$ determined by $e_1$ that contains $-u_1$, and draw them as a ``spray'' with initial point at the second vertex. We claim that $e_2$ is the edge that makes the most obtuse angle with the edge $e_1$, and thus is uniquely determined. We prove this by induction on the number of edges of the polygon. To be more precise, we want to prove 
\begin{equation}\label{most_obtuse_prop}
 \langle e_1,e_b \rangle \leq \langle e_1,e_2\rangle 
\end{equation}
for all $b \in I$, where $I$ is the set of indices corresponding to edges in $\mathcal{H}$. 

The base case for the induction is triangles. The statement holds in this case because there is a single edge that lies in the half-plane $\mathcal{H}$ determined by $e_1$ and $-u_1$; if both edges were in $\mathcal{H}$ then the triangle would not close.

Now assume the statement holds for all convex polygons with $d \geq 3$ edges and let $P$ be a polygon with one vertex at the origin and ordered list of edges $e_1, e_2, \ldots, e_{d+1}$. Move the edge $e_1$ in the direction of $-u_1$, and allow the length of $e_1$ to vary as necessary for the polygon to remain closed. Eventually the number of edges in the polygon will decrease. Let $P'$ denote the resulting polygon; note that the lengths of some of the edges may be different in $P'$ and $P$. One of the following cases occurs.

\noindent \textbf{Case 1:}  the edge $e_{d+1}$ vanishes

\begin{figure}[h]
\centering
\includegraphics[scale=0.9]{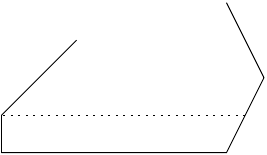}
\caption{The last edge vanishes}
\end{figure}

Let $\mathcal{H}'$ be the half-plane determined by $e_1'$ which contains $-u_1$.  Since $e_{d+1} \in (\mathcal{H'})^C$, the set $I'$ corresponding to $e_1' \in P'$ is the same as the set $I$ corresponding to $e_1 \in P$. By induction $P'$ satisfies property (\ref{most_obtuse_prop}).  The result follows.

\noindent \textbf{Case 2:} the edge $e_2$ vanishes

\begin{figure}[h]
\centering
\includegraphics[scale=0.9]{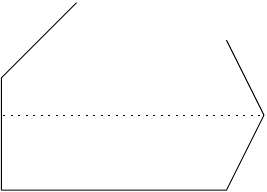}
\caption{The second edge vanishes}
\end{figure}

Since $e_1'$ is parallel to $e_1$, induction gives 
\begin{displaymath}
 \langle e_1,e_b \rangle <\langle e_1,e_3 \rangle
\end{displaymath}
for all $b \in I'$.  Moreover, 
$\langle e_1,e_3 \rangle <\langle e_1,e_2 \rangle$, for
otherwise the angle between $e_2$ and $e_3$ in $P$ would be greater than $\pi$, violating convexity. Therefore we must have
\begin{displaymath}
  \langle e_1,e_2 \rangle >\langle e_1,e_3 \rangle \geq \langle e_1, e_b \rangle
\end{displaymath}
for all $b$ in $I$. Hence property (\ref{most_obtuse_prop}) holds for $P$.

\noindent \textbf{Case 3:} both edges $e_{d+1}$ and $e_2$ vanish simultaneously  

If $d \geq 4$, then $P'$ has at least three edges and we can invoke the induction hypothesis; the arguments made in Cases 1 and 2 can be combined to show that property (\ref{most_obtuse_prop}) holds for $P$.  If $d=3$, then $P'$ consists of $e_1'$ collapsed onto $e_3$.  This implies that $e_1$ and $e_3$ are parallel; combining this with the fact that $e_4 \not \in \mathcal{H}$, we see that $I = \{2\}$ so that property (\ref{most_obtuse_prop}) is trivially satisfied.

Since the choice of $e_1$ was arbitrary, the above argument shows that the edge $e_{k+1}$ in the polytope is always uniquely determined as the edge in the appropriate half-plane making the most obtuse angle with $e_k$.  Thus the ordering of the edges in the polytope is unique.
\end{proof}

We cannot yet apply Lemma \ref{most_obtuse} directly to our data $\mcD$: we do not know the lengths of the edges exactly, and we only know the edge vectors up to sign.  One situation in which we \emph{do} know the lengths of the edges exactly is if there is one edge corresponding to each normal vector in the data $\mcD$.  This occurs when $P$ does not have parallel sides.  Since all Delzant polygons with four or more edges have at least one pair of parallel sides (cf. \S \ref{sec:zoo}), we know the lengths of the edges exactly for Delzant triangles.  In this case, we can easily dispatch with the problem of only knowing the edge vectors up to sign.  There is at least one choice of signs so that the resulting vectors are the edges of a Delzant triangle; we arbitrarily choose an initial vector in this triangle and label it $e_1$.  The subsequent vector is $e_2$ and the last edge vector is $e_3$.  Recalling that $e_1 + e_2 + e_3 = 0$, we see that changing the sign of exactly one or exactly two of our edge vectors corresponds to an edge vector being identically zero.  Thus the allowable sets of edge vectors for a Delzant triangle are $\{e_1,e_2,e_3\}$ and $\{-e_1,-e_2,-e_3\}$.  We will return to the special case of Delzant triangles shortly, but first discuss the more general situation.
  
Suppose that the only allowable sets of edge vectors for a Delzant polygon are $\{e_1, \ldots, e_d\}$ and $\{-e_1, \ldots, -e_d\}$. The number of such polygons with data $\mcD$ is $4d$, since we can
\begin{enumerate}
\item choose the initial vector of the polygon;
\item choose the sign of that vector;
\item choose the sign of the normal to that vector.
\end{enumerate}
We will see that all these choices give rise to polygons which are translates of $P_{e_1,u_1}$ or $P_{e_1,-u_1}$.  

\begin{lemma}\label{4dto2}
Fix the spectral data $\mcD$ and suppose that the only allowable sets of edge vectors corresponding to $\mcD$ are $\{e_1, \ldots, e_d\}$ and $\{-e_1, \ldots, -e_d\}$.  Then every convex polygon corresponding to  data $\mcD$ and with these edges is a translate of $P_{e_1,u_1}$ or $P_{e_1,-u_1}$.  
\end{lemma}

\begin{proof}
We have
\begin{displaymath}
P_{e_1,u_1}=\text{hull}(0, e_1,e_1+e_2,\ldots,e_1+\cdots+e_{d-1}).
\end{displaymath}
It is easy to check that
\begin{displaymath}
P_{e_2,u_2}=\text{hull}(0, e_2,\ldots,e_2+\cdots+e_{d}),
\end{displaymath}
since such a polygon satisfies all the conditions in Lemma \ref{most_obtuse}.  Note that
\begin{displaymath}
\text{hull}(0, e_2,\ldots,e_2+\cdots+e_{d})+e_1=\text{hull}(0, e_1,e_1+e_2,\ldots,e_1+\cdots+e_{d-1}),
\end{displaymath}
since $e_1+\cdots+e_{d}=0$. Therefore $P_{e_2,u_2}$ is a translate of $P_{e_1,u_1}$ as claimed. The same argument shows that $P_{e_l,u_l}$ is a translate of $P_{e_1,u_1}$ for any $l$. 

Next consider what happens when we change the sign of the initial vector.  We have
\begin{displaymath}
P_{-e_1,u_1}=\text{hull}(0, -e_1,-e_1-e_d,\ldots,-e_1-e_{d}-\cdots-e_{3})
\end{displaymath}
since, as one can check, the above hull satisfies all the properties in Lemma \ref{most_obtuse}.  Furthermore,  
\begin{eqnarray*}
\text{hull}(0, -e_1,-e_1-e_d, \ldots,-e_1-e_{d}-\cdots-e_{3})+(-e_2) \hspace*{2cm} \\
=  \text{hull}(e_3 + \cdots + e_d + e_1, e_3 + \cdots + e_d, e_3 + \cdots + e_{d-1}, \ldots, 0) \\
= \text{hull}(0,e_3, \ldots, e_3 + \cdots + e_d, e_3 + \cdots + e_d + e_1 ); \hspace*{1.85cm}
\end{eqnarray*}
this follows from suitably manipulating $e_1 + \cdots + e_d = 0$ to get expressions like $e_3 = -e_1 - e_{d}-e_{d-1}-\cdots-e_{4}-e_2$.  Therefore $P_{-e_1,u_1}$ is a translate of $P_{e_3,u_3}$ and hence of $P_{e_1,u_1}$; again, the same argument shows that $P_{-e_l,u_l}$ is a translate of $P_{e_1,u_1}$ for any $l$.  

Straightforward modifications of the above arguments show that $P_{e_l,-u_l}$ and $P_{-e_l,-u_l}$ are translates of $P_{e_1,-u_1}$ for any $l$.  Thus, up to translation, $P_{e_1,u_1}$ and $P_{e_1,-u_1}$ are the only two convex polygons corresponding to $\mcD$ with edges $e_1, \ldots, e_d$ or $-e_1, \ldots, -e_d$.
\end{proof}

The following proposition now follows immediately.

\begin{prop}\label{prop:triangles}
Let $P$ be a Delzant triangle which corresponds to a fixed set of data $\mcD$.  Then, up to translation, there are exactly two possibilities for $P$ (see Figure \ref{fig:twoposs}).
\end{prop}

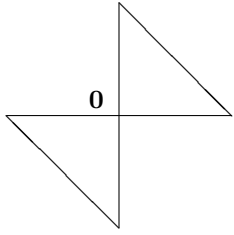
\begin{figure}[h]
\centering
\begin{center}
\setlength{\unitlength}{5mm}
\begin{picture}(11,7)(0,0)
\put(5.5,3){\line(1,0){3}}
\put(5.5,3){\line(0,1){3}}
\put(5.5,6){\line(1,-1){3}}
\put(4.7,3.2){$\mathbf{0}$}
\put(5.5,3){\line(-1,0){3}}
\put(5.5,3){\line(0,-1){3}}
\put(5.5,0){\line(-1,1){3}}
\end{picture}
\end{center}
\caption{Two possibilities for the Delzant triangle $P$}
\label{fig:twoposs}
\end{figure}

We next examine Delzant polygons with one or two pairs of parallel sides.  Note that given the spectral data $\mcD$, we immediately know when the corresponding Delzant polygon $P$ is a rectangle: we hear that there are four vertices and only two normal directions, which means that $P$ consists of exactly two pairs of parallel sides and each edge which is normal to $u_i$ has length $\frac{1}{2}l_i$.  It is not difficult to see that we are again in the setting of Lemma \ref{4dto2}, and there are two possibilities for $P$, up to translation.  For more general Delzant polygons with parallel sides, the arguments are more complicated. 

\begin{prop}\label{2pairs}
 Let $P$ be a generic Delzant polygon in $\bbR^2$ with no more than two pairs of parallel sides.  Then, up to translation, the data $\mcD$ for $P$ determine $P$ up to two choices.
\end{prop}

\begin{proof}
For $P$ with one or two pairs of parallel sides the following indeterminants arise:
\begin{enumerate}
 \item We do not know the individual lengths of the edges in a parallel pair, only the sum of their lengths.
\item We do not know which of the vectors in our set of normal vectors are repeated, i.e., which of the normal vectors are normal to two edges.
\item We only know the edges up to sign.
\end{enumerate}

We begin by addressing the first issue. 
Let $\{e_1,\ldots,e_d\}$ be the set of edge vectors of the polygon. Without loss of generality, we may assume that $e_1$ has a parallel side; in the case of two pairs of parallel sides, assume for notational simplicity that $e_1$ and $e_2$ have parallel sides.  Let $e_{i1}$, and $e_{i2}$ if necessary, be the edge(s) parallel to $e_1$ (and $e_2$). We have
\begin{equation}\label{sum_edges_old}
e_1+e_2+\cdots+e_{i1}+\cdots+e_{i2}+\cdots+e_d=0.
\end{equation}
Suppose there is another polygon with the same spectral data $\mcD$ and same edge vectors, but the lengths of the individual edges in a parallel pair are different. The new edges in this polygon are $(1+\alpha_1)e_1$ and $e_{i1}+\beta_1e_1$ (and $(1+\alpha_2)e_2$ and $e_{i2}+\beta_2e_2$, if necessary).  All other edges remain the same. So we must have
\begin{displaymath}
(1+\alpha_1)e_1+(1+\alpha_2)e_2+\cdots+e_{i1}+\beta_1e_1+\cdots+e_{i2}+\beta_2e_2+\cdots+e_d=0.
\end{displaymath}
But the spectral data tells us that the sum of the lengths of two parallel sides must be the same in the new polygon as in the original. In the new polygon, one such sum is 
\begin{displaymath}
|(1+\alpha_1)e_1-e_{i1}-\beta_1e_1|=|e_1-e_{i1}|+|\alpha_1-\beta_1||e_1|,
\end{displaymath}
and in the original the corresponding sum is $|e_1-e_{i1}|$; thus $\alpha_1=\beta_1$.  In the case of two pairs of parallel sides, we would also obtain $\alpha_2=\beta_2$. In fact we have
\begin{equation}\label{sum_edges_new}
(1+\alpha_1)e_1+(1+\alpha_2)e_2+\cdots+e_{i1}+\alpha_1e_1+\cdots+e_{i2}+\alpha_2e_2+\cdots+e_d=0.
\end{equation}
Subtracting equation (\ref{sum_edges_new}) from equation (\ref{sum_edges_old}) leads to
\begin{displaymath}
\alpha_1e_1+\alpha_2e_2=0,
\end{displaymath}
which implies that $\alpha_1 = \alpha_2 = 0$ since $\{e_1, e_2\}$ is a basis for $\bbR^2$.

Now we address the second issue. Suppose we are given a set of spectral data $\mcD$. 
It is not necessarily the case that there is only one choice of repeated normals which corresponds to a valid polygon $P$.  However, we will show that there is another valid polygon arbitrarily close to $P$ with the same number of repeated normals as $P$ for which there \emph{is} a unique choice of repeated normals. 

Without loss of generality, we may assume that $P$ is such that the normal to $e_1$, say $u_1$, is repeated; in the case of two pairs of parallel sides, suppose $u_2$ is also repeated.  Using the same notational conventions as above, we have that equation (\ref{sum_edges_old}) holds.  If there is another polygon $P'$ with the same spectral data as $P$ and different repeated normals we will assume for notational convenience that the new repeated normals are $u_3$ and $u_4$. Since the length of $e_1' \in P'$ must equal the sum of the lengths of $e_1$ and $e_{i1}$, we have $e_1' =  \pm(e_1-e_{i1})$ (and $e_2' = \pm (e_2-e_{i2}$)). Also, $e_3' = \alpha e_3$ and $e_{i3}' = (\alpha-1) e_3$ since $|e_3' -e_{i3}'| = |e_3|$ (and $e_4' = \beta e_4$, $e_{i4}' = (\beta-1)e_4$).  We have 
\begin{equation}\label{sum_edges_new_2}
\pm (e_1-e_{i1}) \pm(e_2-e_{i2})+\alpha e_3+\beta e_4 \pm e_5 \pm \cdots + (\alpha-1) e_3 \pm \cdots + (\beta-1) e_4 \pm\cdots \pm e_d=0,
\end{equation}
where the $\pm$ signs come from the fact that we only know the edges of $P'$ up to sign.  Adding equations (\ref{sum_edges_old}) and (\ref{sum_edges_new_2}) gives
\begin{equation}\label{non_generic}
a_1e_1+a_2e_2+ \alpha e_{3}+ \beta e_{4}+\sum_{i\in I}e_i=0,
\end{equation}
where $a_i \in \bbR$ account for $e_{i1}$ and $e_{i2}$ as necessary and $I$ is a certain subset of indices in $\{5,\ldots,d\}$.  Namely, $I$ consists of the indices on the edges which have the same sign in $P$ and $P'$.  We can assume that $I \subsetneq \{5,\ldots,d\}$ by using the negative of equation (\ref{sum_edges_old}) to replace the sum in (\ref{non_generic}) by 
\begin{equation}\label{bigI}
a_1'e_1 + a_2'e_2 + (\alpha-1) e_{3}+ (\beta-1) e_{4} = 0
\end{equation}
if $I$ is as large as possible.

Now choose $e_i$ in equation \eqref{non_generic} or \eqref{bigI} such that $e_{i+1}$ is not in the sum and such that $e_{i-1}$ and $e_{i+1}$ are not parallel; this is always possible since we have treated the case of rectangles separately. Perturb the $\lambda_i$ corresponding to $e_i$ in $P$ by a small amount. This will modify the relevant sum by $c_ie_i+c_{i-1}e_{i-1}$, and we can choose the perturbation so that the sum is no longer zero. Note that we will have to do this for each possible choice of repeated normal, but there are only finitely many such choices; moreover, we may choose the change in $\lambda_i$ sufficiently small that a set of edges whose sum was nonzero keeps a nonzero sum under the perturbation.  So we can ensure that our perturbed polygon satisfies none of the relevant equations (\ref{non_generic}) or \eqref{bigI}.  Since varying $\lambda_{i}$ does not interfere with the Delzant condition, we have found a valid polygon arbitrarily close to $P$ with the same number of repeated normals as $P$ and with a unique choice of these normals.

Note that this argument addressing repeated normals can also be used to deal with the sign ambiguity for the edges.  For the convenience of the reader, we make this more explicit.  We say that $P$ has a subpolygon if there exists a proper subset of $\{1, \ldots, d\}$, say $\{i_1, \ldots, i_k\}$ with $k \geq 3$, such that $e_{i_1}+ \cdots+ e_{i_k}=0$; we also require that the complement of $\{i_1, \ldots, i_k\}$ in $\{1, \ldots, d\}$ contains at least three elements.  The subset of vectors with changed signs gives rise to a closed polygon, and these vectors must be non-consecutive in our larger polygon.  We may again encode this behavior in an equation like \eqref{non_generic}, choose an $e_i$ as above, and perturb the $\lambda_i$ corresponding to $e_i$ so that the sum is no longer zero.  By varying $\lambda_i$, we decrease the number of closed subpolygons.  
 
Thus for a Delzant polygon $P$ with no more than two pairs of parallel sides, ``generic'' will mean that $P$ does not contain subpolygons and has a unique choice of repeated normals.  Lemma \ref{4dto2} then implies that there are two choices for $P$, up to translation.
\end{proof}

Finally, we examine the case when $P$ has three pairs of parallel sides. It is only at this stage that we use that the volume of $P$ is determined by $\mcD$. 
\begin{prop}\label{prop:3pairs}
Let $P$ be a generic Delzant polygon with no more than three pairs of parallel sides.  Then, up to translation, the data $\mcD$ determines $P$ up to at most four possibilities.
\end{prop}
\begin{proof}
If $P$ has at most two pairs of parallel sides we have seen that the first three items in $\mcD$ determine $P$ up to two possibilities. So we assume that $P$ has three pairs of parallel sides.  Note that the edge vectors which are not in a parallel pair are determined up to sign, while the lengths of the individual edges in a parallel pair are a priori not determined. We will show that for every choice of edge signs and choice of repeated normals there are in fact only two choices for these lengths. Then we proceed as in the proof of Proposition \ref{2pairs}: if necessary, we perturb $P$ slightly a finite number of times to avoid situations which would correspond to different choices of signs for the edges or different choices of repeated normals. Since there are only a finite number of choices for the lengths of the parallel edges, we have a finite number of ``bad" cases to perturb away. 

Let $\{e_1,\ldots,e_d\}$ be the set of edge vectors of $P$.  For notational simplicity, we assume that the edges which have a parallel partner are $e_1,e_2,$ and $e_3$.  Let $e_{i1}, e_{i2}, e_{i3}$ be the edges parallel to $e_1, e_2, e_3$, respectively.  We have seen that for a generic polygon $P$ the ordering of the edges, the signs of the normals, and the collection of repeated normals is completely determined, once we choose an initial vector and the sign of its normal. 
Therefore another polygon with the same data $\mcD$ has the same edges as $P$ except the edges $\{e_1,e_2,e_3, e_{i1}, e_{i2},e_{i3}\}$ may be replaced by 
\begin{displaymath}
\{e_1-\alpha_1e_1,e_2-\alpha_2e_2,e_3-\alpha_3e_3, e_{i1}-\alpha_{1}e_{1}, e_{i2}-\alpha_{2}e_{2},e_{i3}-\alpha_{3}e_{3}\};
\end{displaymath}
note that if we subtract $\alpha_1e_1$ from the edge $e_1$, then we must subtract the same quantity from $e_{i1}$ in order for the sum of the lengths to be preserved (cf. Proposition \ref{2pairs}).  We must have 
\begin{displaymath}
e_1-\alpha_1e_1+e_2-\alpha_2e_2+e_3-\alpha_3e_3+ e_{i1}-\alpha_{1}e_{1}+ e_{i2}-\alpha_{2}e_{2}+e_{i3}-\alpha_{3}e_{3}+\sum_{i \in I} e_i=0
\end{displaymath}
where $I$ is the set of indices in $\{1, \ldots, d\}$ which are distinct from $\{1,2,3,i1,i2,i3\}$.  We also have
\begin{displaymath}
e_1+e_2+e_3+\cdots+ e_{i1}+\cdots+ e_{i2}+\cdots+e_{i3}+\cdots+e_d=0.
\end{displaymath}
Therefore
\begin{displaymath}
\alpha_1e_1+\alpha_2e_2+\alpha_3e_3=0.
\end{displaymath}
We see that all polygons with the same data $\mcD$, same choice of sign for the edge vectors, and same choice of repeated normals as $P$ come from a linear relation among the vectors $e_1,e_2$, and $e_3$. 
Figure~\ref{fig:Pt} shows a polygon $P_0$ with three pairs of parallel sides and a modification $P_t$ of $P_0$ with the same normals and sums of lengths.

\begin{figure}[h]
\centering
\subfigure[$P_0$]
{\includegraphics[scale=0.75]{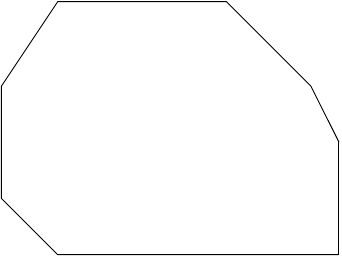}}
\hspace{2cm}
\subfigure[$P_t$]
{\includegraphics[scale=0.75]{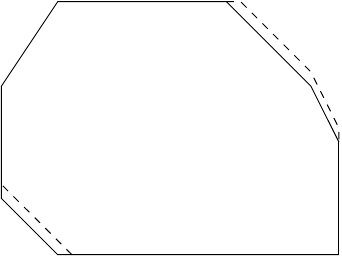}}
\caption{$P_0$ and $P_t$}
\label{fig:Pt}
\end{figure}

We will show that $\text{vol}(P_t)$ is a degree two polynomial in $t$. To this end, fix $t$ and consider the quantity $\text{vol}(P_t)-\text{vol}(P)$. This is in fact a (signed) sum of volumes of parallelograms and trapezoids as in Figure~\ref{fig:modifying}, where the solid lines are edges of $P$ and the dashed lines are edges of $P_t$.

\begin{figure}
\centering
\includegraphics[scale=0.7]{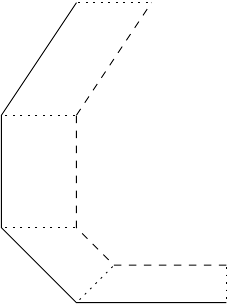}
\caption{A modified piece of polygon}
\label{fig:modifying}
\end{figure}

\begin{itemize}
\item For each parallelogram in the sum, only the length of one of the sides depends on $t$; it is in fact proportional to $t$. The angles of the parallelogram as well as the lengths of the sides parallel or equal to an edge in $P$ do not depend on $t$. So the volume of a parallelogram is of the form $At$, where $A$ only depends on $P$. 
\item As for the trapezoids whose volumes appear in the sum, their angles are fixed.  The length of the side which is an edge in $P$ is also fixed, while the lengths of the sides transversal to this side are again proportional to $t$. The length of the side parallel to an edge in $P$ is of the form $l+ct$ where $l$ and $c$ are constants independent of $t$. So the volume of such a trapezoid is of the form $At+Bt^2$ where  $A$  and $B$ only depend on $P$.
\end{itemize}
Thus $\text{vol}(P_t)=\text{vol}(P)+At+Bt^2$ for some constants $A$ and $B$, and there is at most one nonzero value of $t$ for which $\text{vol}(P_t)=\text{vol}(P)$.  Hence there are at most two choices for the lengths of the individual edges in a parallel pair.  Taking the same definition of generic as in Proposition \ref{2pairs} and applying Lemma \ref{4dto2} to each of the two possible sets of edge vectors, we see that up to translation there are at most four possibilities for $P$.
\end{proof}

\begin{rem}
When the number of parallel pairs for the polygon $P$ is $p\geq 3$, the family of polygons with data $\mcD$ is a priori of dimension $p-2$. There is one parameter per linear relation among the parallel edges. The volume of each polygon in such a family is a polynomial of degree $2$ in each of the parameters. The condition that the volume is fixed reduces the number of degrees of freedom by one, leaving a $p-3$ parameter family of polygons with data $\mcD$. For $p \geq 4$, there is no longer a finite number of polygons determined by $\mcD$.
\end{rem}

\section{Zoology}\label{sec:zoo}

The goal of this section is to discuss the set of Delzant polygons to which our theorems apply. 
We have seen that our methods can be applied to all Delzant triangles and generic Delzant polygons in $\bbR^2$ without too many pairs of parallel sides, so it makes sense to ask if such polygons are ``frequent." 
\begin{thm} \label{open_one_pair_parallel}
The set of all Delzant polygons in $\bbR^2$ with $d \geq 5$ sides and one pair of parallel sides is a nonempty, proper open set in the set of all Delzant polygons in $\bbR^2$. The same holds for the set of all Delzant polygons with at most three pairs of parallel sides.
\end{thm}
We prove this theorem below.  A priori one could hope to prove an analogous theorem for Delzant polygons without parallel sides. In fact, all Delzant polygons with four or more sides have at least one pair of parallel sides.  To explain this claim, we begin by recalling that a Hirzebruch surface is a K\"ahler toric surface. It is in fact $\mathbb{P}(\mathcal{O}\otimes\mathcal{O}(r))$, the projectivization of the vector bundle $\mathcal{O}\otimes\mathcal{O}(r)$ over $\bbC \mathbb{P}^1$ for some integer $r$. For our purposes it is enough to draw a picture of the moment polygons of Hirzebruch surfaces:

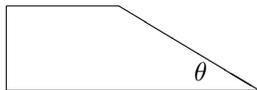
\begin{figure}[h]
\centering
\begin{center}
\setlength{\unitlength}{5mm}
\begin{picture}(11.25,3)(0,0)
\put(.5,.5){\line(1,0){6.75}}
\put(.5,.5){\line(0,1){2.25}}
\put(.5,2.75){\line(1,0){3}}
\put(3.5,2.75){\line(5,-3){3.7}}
\put(5.5,.75){$\theta$}
\end{picture}
\end{center}
\caption{The moment polygon of a Hirzebruch surface with $\tan \theta = \frac{1}{r}$, where $r$ ranges over the nonnegative integers}
\end{figure}

\noindent Note that if $d=4$, knowing the number of repeated normals allows us to construct $P$, up to translation and two choices; $P$ is either a parallelogram or the moment polygon of a Hirzebruch surface.

In fact moment polygons of Hirzebruch surfaces are in some sense the building blocks from which we can get all Delzant polygons.
\begin{thm}\label{corner_chopping}
Given a Delzant polygon $P$ with $d \geq 5$ sides, there is an $SL(2,\bbZ)$ transformation that takes $P$ into a corner chopping of a Delzant polygon of a Hirzebruch surface.
\end{thm}
The corner chopping of a convex polytope $P$ at a subset of the set of vertices of $P$ is a polytope $P'$ which is obtained from $P$ by deleting a neighborhood of each vertex in the subset and replacing it with the convex hull of what is left. In the following picture we have chopped the moment polygon of a Hirzebruch surface at one vertex.

\begin{figure}[h]
\centering
\includegraphics[scale=0.60]{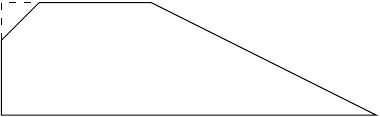}
\caption{A chopped moment polygon}
\end{figure}

It is well known that among Delzant polytopes this operation corresponds to a symplectic blow-up of the toric manifold associated with the Delzant polytope.  See \cite{Fulton} for a proof of Theorem \ref{corner_chopping} and \cite{kkp} for further discussion. Since an $SL(2,\bbZ)$ transformation preserves the number of parallel sides and the Delzant polygon of a Hirzebruch surface has parallel sides, we conclude that all Delzant polygons with $5$ or more sides have at least one pair of parallel sides. 

\begin{lemma}
Given a Delzant polygon $P$ in $\bbR^2$ and a vertex $v$ in $P$ there is a $1$-parameter family of choppings of $P$ at $v$.
\end{lemma}
\begin{proof}
The proof is very simple. Using an $SL(2,\bbZ)$ transformation if necessary, we may assume that $v = 0 \in \bbR^2$ and that the two edges of $P$ meeting at $v$ lie along the positive coordinate axes.  Now let $a$ and $b$ be the two new angles in the triangle determined by the chopping of $P$ at $v$. 

\begin{figure}[h]
\centering
\includegraphics[scale=0.80]{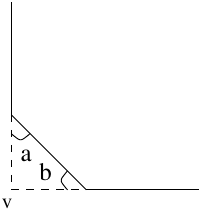}
\end{figure}

These two angles sum to $\frac{\pi}{2}$, and the Delzant condition implies that there are integers $m$ and $n$ such that
\begin{displaymath}
\tan(a)=\frac{1}{m}, \:\:\: \tan(b)=\frac{1}{n}.
\end{displaymath}
Now
\begin{displaymath}
\tan(a+b)=\frac{\tan(a)+\tan(b)}{1-\tan(a)\tan(b)}
\end{displaymath}
and $\tan(a+b) = \tan(\frac{\pi}{2}) = \infty.$
So $\tan(a)\tan(b)=1$ which implies $mn=1$, and thus the angles $a$ and $b$ are both $\frac{\pi}{4}$. The one parameter corresponds to the length of the new side.
\end{proof}

We now prove Theorem \ref{open_one_pair_parallel}.
\begin{proof}
Given the characterization of Delzant polygons provided by Theorem \ref{corner_chopping}, we see that the space of all Delzant polygons with $d$ edges is parametrized by
\begin{displaymath}
\{(A,H, \alpha,l): A\in SL(2,\bbZ), H \in \mathcal{B}, \alpha \in \mathcal{C}_{H}, l \in \bbR^{d-4}\},
\end{displaymath}
where $\mathcal{B}$ is the set of all Delzant polygons corresponding to Hirzebruch surfaces, $\mathcal{C}_{H}$ is the set of all sequences of positions at which one may chop to get a Delzant polygon from $H$, and $l$ records the lengths of the $d-4$ chopped edges associated to $\alpha$.
Note that for a fixed $d$, $\mathcal{C}_{H}$ is a finite set.

Polygons with more than one pair of parallel edges correspond to specific choppings, i.e., to specific sequences $\alpha$. The set of such polygons is parametrized by
\begin{displaymath}
\{(A,H, \alpha,l): A\in SL(2,\bbZ), H \in \mathcal{B}, \alpha \in \mathcal{C}^1_{H}, l \in \bbR^{d-4}\}
\end{displaymath}
where $\mathcal{C}^1_{H}$ is the subset of $\mathcal{C}_{H}$ which gives rise to more pairs of parallel sides. We claim that $\mathcal{C}^1_{H}$ is a nonempty proper subset of $\mathcal{C}_{H}$. It is proper because there are choices for the chopping data which will not give rise to more pairs of parallel sides. As one example, consider all choppings involving only the right angle of $H$ which is opposite the acute angle of $H$:

\vspace{.1in}
\hspace*{1.5in}
\includegraphics[scale=0.60]{chopping.eps}
\hspace{1in}
\\

One can see by inspection that there are many other examples. Since there are $\alpha$'s for which there \emph{are} more pairs of parallel sides introduced through corner chopping, $\mathcal{C}^1_{H}$ is nonempty. For example, every $\alpha$ which involves chopping opposite corners of a Delzant rectangle $H$ introduces at least one more pair of parallel sides.  Hence both $\mathcal{C}^1_{H}$ and the complement of $\mathcal{C}^1_{H}$ in $\mathcal{C}_{H}$ are proper, nonempty subsets of $\mathcal{C}_{H}$. 
\end{proof}
The proof above shows that a random Delzant polygon with $d$ sides has a positive probability $p(d)$ of having no more than three pairs of parallel sides. It also shows that $p(d)$ is not one. In fact, $p(d)$ is the number of choppings which give rise to no more than three pairs of parallel sides divided by the total number of choppings which give rise to a polygon with $d$ edges. 

We conclude our discussion of choppings of Delzant polygons by proving Proposition \ref{number_facets}, which tells us that we can recover the number of vertices of a Delzant polygon from the spectrum of the real manifold naturally associated to a toric surface.

\begin{proof}
Let $M^4$ be a symplectic toric manifold, and let $d$ be the number of sides of the associated Delzant polygon $P$.  Theorem \ref{corner_chopping} tells us that if $d \geq 5$ then $M$ is symplectomorphic to a blow up of $\mathbb{CP}^2$, i.e., it is symplectomorphic to a connected sum
\begin{displaymath}
 \mathbb{CP}^2\#(d-3)\mathbb{CP}^2.
\end{displaymath}
Since the real manifold associated to $\mathbb{CP}^2$ is $\mathbb{RP}^2$ we see that the real manifold $M_\bbR$ is diffeomorphic to
\begin{displaymath}
 \mathbb{RP}^2\#(d-3)\mathbb{RP}^2, 
\end{displaymath}
which has Euler characteristic $4-d$.  As one can check directly, this expression for the Euler characteristic also holds when $d=3$ or $d=4$.  For a real $2$-dimensional manifold, the spectrum determines the Euler characteristic; this is a simple consequence of the Gauss-Bonnet theorem and the fact that the spectrum determines the integral of the scalar curvature via the asymptotic expansion of the usual heat trace (e.g., \cite[p. 222]{BGM}).  Hence the spectrum of $M_{\bbR}$ determines $d$, and thus determines the number of vertices of $P$.
\end{proof}

\section{The spectrum of a canonical line bundle}\label{sec:bundle}

In this section we will give a complete and relatively short argument answering a stronger version of Question \ref{q:ours}. Assume that $(M,\omega)$ is a toric manifold and the cohomology class of $\omega$ has integer coefficients, i.e., $[\omega]\in H^2(M,\bbZ)$. There is a line bundle $L$ over $M$ naturally associated to this data, namely a line bundle such that $c_1(L)=[\omega]$. There is also a connection on $L$ whose curvature is $\omega$, and the connection allows us to define a Laplace operator $\triangle_L$ on $\mathcal{C}^{\infty}(L)$. As an example, one may consider the toric manifold associated to a Delzant polytope with vertices in $\bbZ^n$ and endowed with the reduced metric.  When we consider the equivariant spectrum of the Laplacian acting on smooth sections of $L$, we are able to recover the Delzant polytope exactly.   

\begin{thm} 
The equivariant spectrum of the Laplacian $\triangle_L:\mathcal{C}^{\infty}(L)\rightarrow \mathcal{C}^{\infty}(L)$ determines the Delzant polytope of $M$.
\end{thm}

\begin{proof}
As usual $\phi$ denotes the moment map of the $\bbT^n$-action on $M$ and $P=\phi(M)$ is the Delzant polytope of $M$. Let $p$ be a fixed point for the $\bbT^n$-action and $v=\phi(p)$ a vertex in $P$. There is a $\bbT^n$-action on the total space of $L$ and therefore an induced action on the space of smooth sections of $L$, denoted $\mathcal{C}^\infty(L)$. The infinitesimal action associated to this induced action is given by Kostant's formula:
\begin{displaymath}
 \mathcal{L}_{X_\theta} s=\triangledown_{X_\theta}s+i\phi\ip\theta s,
\end{displaymath}
where $s \in \mathcal{C}^\infty(L)$, $\theta \in \bbR^n$, and $X_\theta$ is the vector field in $M$ induced by the $\bbT^n$-action. From this we see that the weight of the isotropy representation of $\bbT^n$ on the fiber of $L$ over $p$ is $\phi(p)$.

We make the following notational conventions:
\begin{itemize}
\item $e_1,\ldots, e_n$ are the edges of $P$ meeting at $v$,
\item $F_i$ is the hyperplane spanned by $\{e_1,\cdots,\hat{e_i},\cdots, e_n\}$,
\item $u_i$ is the outward normal to $F_i$ in $\bbR^n$,
\item $G_i$ is the one-parameter subgroup of $\bbT^n$ generated by $u_i$, i.e., the Lie algebra of $G_i$ is spanned by $u_i$,
\item $\beta_i$ is the weight of the isotropy representation of $G_i$ on the normal bundle to $\phi^{-1}(F_i)$ in $M$,
\item and $\xi_i$ is an element of the Lie algebra of $G_i$ such that $\beta_i(\xi_i)=1$.
\end{itemize}
As we have seen in Lemma \ref{fixed_point_set}, the fixed point set of each $G_i$ contains the pre-image via the moment map of the facets which are perpendicular to $u_i$; hence the fixed point set has at most two connected components of nonzero dimension. Fix $i \in \{1,\ldots,n\}$. We will consider two cases.

\noindent \textbf{Case 1:} Assume there is only one facet perpendicular to $u_i$, say $F_i$. This is the simplest case. The coefficient of the leading term in the heat trace corresponding to the $e^{i\xi}\in G_i$ action on $L$ is
\begin{displaymath}
\frac{e^{i\phi(p)\ip\xi}\text{vol}(F_i)}{2-2\cos(\beta_i(\xi))};
\end{displaymath}
setting $\xi=t\xi_i$ gives
\begin{displaymath}
\frac{e^{itv\ip\xi_i}\text{vol}(F_i)}{2-2\cos(t)}.
\end{displaymath}
Hence $\text{vol}(F_i)$ and $e^{sv\ip\xi_i}$, for any $s \in \mathbb{R}$ are spectrally determined.  Thus we know the quantity $v\ip\xi$ for any $\xi$ which is a multiple of $u_i$ and the spectrum determines the hyperplane containing the facet $F_i$, namely
\begin{displaymath}
\{x\in \bbR^n:x\ip u_i=v\ip u_i\}.
\end{displaymath} 

\noindent \textbf{Case 2:} Now assume that there are two facets perpendicular to $u_i$, say $F_+$ and $F_-$. Let $v_+$ and $v_-$ be vertices on $F_+$ and $F_-$, respectively. The coefficient of the leading term in the heat trace corresponding to the $e^{i\xi}\in G_i$ action on $L$ is
\begin{displaymath}
\frac{e^{iv_+\ip\xi}\text{vol}(F_+)+e^{iv_-\ip\xi}\text{vol}(F_-)}{2-2\cos(\beta_i(\xi))}.
\end{displaymath}
Again one can set $\xi=t\xi_i$ and the above formula becomes
\begin{displaymath}
\frac{e^{itv_+\ip\xi_i}\text{vol}(F_+)+e^{itv_-\ip\xi_i}\text{vol}(F_-)}{2-2\cos(t)}.
\end{displaymath}
Therefore $v_+\ip\xi$ and $v_-\ip\xi$ are spectrally determined for any $\xi$ which is a multiple of $u_i$, and so are the sets
\begin{displaymath}
\{x\in\bbR^n:x\ip u_i=v_+\ip u_i\},
\end{displaymath}
and
\begin{displaymath}
\{x\in\bbR^n:x\ip u_i=v_-\ip u_i\}.
\end{displaymath}

We can assume without loss of generality that our polytope has center of mass at the origin. The spectrum determines the hyperplane containing a facet $F$ and the ``in'' direction determines the half-space $\mathcal{H}_{F}$.  Hence we know
\begin{displaymath}
P=\bigcap_F \mathcal{H}_{F}.
\end{displaymath}
\end{proof}

\section{Concluding Remarks}

One may ask to what extent our results are optimal.  We note that the probabilites $p(d)$ mentioned in \S \ref{sec:zoo} could be calculated explicitly, providing a more concrete idea of the ``size'' of the set of Delzant polygons to which Proposition \ref{prop:3pairs} applies.  Propositions \ref{prop:3pairs}, \ref{2pairs} and \ref{prop:triangles} conclude that the equivariant spectrum determines the Delzant polygon up to a small number of possibilities; in Propositions \ref{prop:triangles} and \ref{2pairs}, the two possibilities correspond to a set of weights for the torus action and the set of conjugate transposes of these weights.  Can these two possibilities be distinguished using spectral data?  Regarding our genericity assumptions, it would be nice to have examples of symplectic toric manifolds (and their corresponding Delzant polygons) which show that these assumptions are necessary.  That is, can one construct pairs of (non-isometric) toric manifolds with the same equivariant spectrum whose Delzant polygons are either non-generic or have many pairs of parallel sides?

One obstacle to generalizing the results of \S\S \ref{sec:comb} and \ref{sec:zoo} to higher-dimensional polytopes is the lack of knowledge of the number of vertices of the Delzant polytope.  
For a toric manifold $M^{2n}$, knowledge of the number of vertices is equivalent to knowledge of any one of the following three quantities: the Euler characteristic of $M$, the Lefschetz number of a torus-induced isometry of $M$, and the number of critical points of a generic component of the moment map.  By considering the equivariant spectrum corresponding to even $p$-forms for $0<p \leq n$, one could recover the number of vertices; in fact, one only needs the multiplicity of the $0$-eigenspaces.
If one fixes the number of vertices and assumes that there are no repeated normals, our methods can be applied.  In particular, there are ``many'' Delzant polytopes in dimension three with four vertices and no repeated normals which can be distinguished by the equivariant spectrum, up to translation and a small number of possibilities.  It is also possible that using more terms in Donnelly's asymptotic expansion \eqref{eqn:don} could lead to stronger results; however, the presence of normals to faces of varying codimension hinders calculation of these higher-order terms in general.  

We should mention again that if the Delzant polygon of $M$ is generic and doesn't have too many pairs of parallel sides, then our results imply that the symplectomorphism type of $M$, and hence its diffeomorphism type, are completely determined by the equivariant and real spectra. In fact, if we know a priori that $M$ is endowed with the reduced metric, the metric on $M$ is also determined by the spectra since the polygon determines a unique reduced metric. Of course one can ask what can be said about the metric on $M$ in general.  For example, one can ask whether the real and equivariant spectra determine if the metric on $M$ is extremal in the sense of Calabi.

The relatively strict constraints on perturbation imposed by the Delzant conditions provide
another obstacle to generalizing our results to all Delzant polytopes.  These constraints are significantly relaxed for generic toric orbifolds, which correspond to rational polytopes with a labelling that encodes the local group actions.  Thus the inverse problem is to recover both the polytope and its labels from the equivariant spectrum; the authors are working on a generalization of this work to that setting.

\bibliographystyle{plain}
\bibliography{inv_spec}

\end{document}